\documentclass[12pt]{amsart}

\usepackage{amssymb}

\usepackage{enumerate}

\usepackage{graphicx}

\makeatletter
\@namedef{subjclassname@2010}{%
  \textup{2010} Mathematics Subject Classification}
\makeatother

\newtheorem{thm}{Theorem}[section]
\newtheorem{cor}[thm]{Corollary}
\newtheorem{lem}[thm]{Lemma}
\newtheorem{prop}[thm]{Proposition}

\theoremstyle{definition}
\newtheorem{defin}[thm]{Definition}
\newtheorem{rem}[thm]{Remark}

\numberwithin{equation}{section}

\newcommand{\abs}[1]{\lvert#1\rvert}

\begin{document}


\baselineskip=17pt



\title[The Bergman representative map]{A differential-geometric analysis of the Bergman representative map}

\author[S. Yoo]{Sungmin Yoo}
\address{Department of Mathematics\\ Pohang University of Science and Technology\\
37673, Pohang, Republic of Korea}
\email{sungmin@postech.ac.kr}


\thanks{This research was supported by Global Ph.D. Fellowship Program (NRF-2011-0006432) and the SRC-GAIA (NRF-2011-0030044) through the National Research Foundation of Korea (NRF) funded by the Ministry of Education.}

\date{}

\begin{abstract}
We show that the exponential map of the Bochner connection on the restricted holomorphic tangent bundle of a complex manifold admitting the positive-definite Bergman metric coincides with the inverse of Bergman's representative map. We also present a generalization of Lu Qi Keng's theorem, as an application.
\end{abstract}



\maketitle

\section{Introduction}
On a complex manifold equipped with the Bergman kernel and metric, the Bergman representative map, originally named as ``the representative domain'' by Stefan Bergman himself, is an important offshoot of the Bergman kernel form (cf. \cite{GKK2011}, Chapter 4). It is a special holomorphic map which is in a significant contrast with the exponential map of the Riemannian structure given by the real part of the Bergman metric; the Riemannian exponential map is almost never holomorphic. On the other hand, the representative map gives rise to a holomorphic K\"{a}hler normal coordinate system with respect to the Bergman metric. One of its best known features is that all holomorphic Bergman isometries become linear mappings in these representative coordinates. In spite of the difficulty that this map need not be well-defined globally, this feature has been proven to be useful in many important work (see for instance, \cite{Lu1966}, \cite{bell1980}, \cite{Webster1979}, \cite{greene1985}, et al.). However, it was striking to us that no systematic study of this concept has yet been carried out. 

In this paper, we describe the relationship between the Bergman representative map and the Riemannian exponential map of the Bergman metric in terms of Differential Geometry. In particular, we present a construction of the torsion-free flat holomorphic affine connection on the holomorphic tangent bundle of an open dense subdomain of the given complex manifold, whose affine exponential map is the inverse to the representative map (Theorem \ref{connection}). This yields a differential geometric interpretation of the Bergman representative map.

It is worth mentioning that our connection was discovered, at least partially, by several other authors in the articles preceding this paper, even though the information was scattered around in the papers such as \cite{bochner1947} (much earlier than the others; in fact, Bochner constructed ``normal'' coordinates only, which can develop into the connection), \cite{calabi1953isometric}, and \cite{bcov1994}. It is also studied independently in \cite{Demailly1982} and \cite{kapranov1999} in relation to the holomorphic part of the K\"{a}hler metric connection (a symplectic geometric interpretation can be found in \cite{kontsevich1995} and \cite{ruan98}). More notably, as the connection for the case of ``bounded domains'', it was studied in \cite{Webster1979} for a version of extension theorem for biholomorphic mappings. We hope that this paper shows these concepts in a unified way.

This paper is organized as follows: First, we briefly review fundamentals of Bergman geometry including the construction of the concept of the representative map. Then we present Bochner's normal coordinate system for the real analytic K\"{a}hler manifolds and the affine connection. We would like to call it {\it the Bochner connection}. Then, we restrict ourselves to complex manifolds with the Bergman metric, and study the Bochner connnection. In Section 5 and 6, we analyze the geodesic behavior of the Bochner connection and the removed varieties. In the last section, as an application, we present a generalization of the theorem by Lu Qi-Keng \cite{Lu1966}, which says that a bounded domain in $\mathbb{C}^n$ whose Bergman metric is complete and of constant holomorphic sectional curvature is biholomorphic to the unit ball. We were able to generalize this to the case of bounded domains with a \textit{pole of the Bochner connection} such as a circular domain or a homogeneous domain.


\section{Fundamentals of Bergman geometry}


\subsection{The Bergman kernel and metric for a bounded domain in $\mathbb{C}^n$}

Let $\Omega$ be a bounded domain in $\mathbb{C}^n$ and $K(z,\overline{w})$ the Bergman kernel of $\Omega$. Since $K(z,\overline{z})>0$, the Bergman metric 
$$
g_{\Omega}(z)=\sum\limits_{j,k=1}^n g_{j\overline{k}}(z) dz_j\otimes d\overline{z_k} \text{\ \ \ \ with\ \ } g_{j\overline{k}}(z)=g_{j\overline{k}}(z,\overline{z}):=\frac{\partial^2\log K(z,\overline{z})}{\partial z_j\partial \overline{z_k}}
$$
is well-defined. In fact, the following result was proved by Bergman himself \cite{bergman1970kernel}:

\begin{thm}[Bergman]
The Bergman metric $g_{\Omega}$ is positive-definite at every $z \in\Omega$.
\end{thm}

\begin{rem}
Note that $g_{\Omega}$ is a K\"{a}hler metric. The transformation formula for the Bergman kernel function (under biholomorphisms) implies that every biholomorphism between bounded domains is an isometry with respect to the Bergman metric.
\end{rem}


\subsection{The Bergman representative map}\label{br}

Let $p$ be a point of $\Omega$. Since $K(p,\overline{p})>0$, there is a neighborhood of $p$ such that $K(z,\overline{w})\neq 0$ for all $z,w$ in that neighborhood. Denote by ${g}^{\overline{k}j}(p)$ the $(k,j)$-th entry of the inverse matrix of $(g_{j\overline{k}}(p))$.

\begin{defin}\label{rep}
The {\it Bergman representative map} at $p$ is defined by 
$$
\hbox{rep}_p(z)=(\zeta_1(z),\ldots,\zeta_n(z)),
$$
where: 
$$
\zeta_j(z):={g}^{\overline{k}j}(p)\Big\{\frac{\partial}{\partial \overline{w_k}}\Big|_{w=p}\log K(z,\overline{w})-\frac{\partial}{\partial \overline{w_k}}\Big|_{w=p}\log K(w,\overline{w})\Big\}.
$$
\end{defin}

Since $\frac{\partial \zeta_k}{\partial z_l}|_{z=p}=\delta_{lk}$, this map defines a holomorphic local coordinate system at $p$. Another special feature is in the following theorem by Bergman himself.

\begin{thm}[Bergman]\label{clin}
If $f:\Omega\rightarrow\tilde{\Omega}$ is a biholomorphic mapping of bounded domains, then $\hbox{\rm rep}_{f(p)}\circ f\circ \hbox{\rm rep}_p^{-1}$ is $\mathbb{C}$-linear.
\end{thm}

The original proof of this by Bergman was via a direct computation using the transformation formula. On the other hand, a differential geometric proof using the Bochner connection will be presented in Section \ref{bb} (see Theorem \ref{connection} as well as Remark \ref{remc}). Since the Bergman kernel and metric can be defined for complex manifolds \cite{kobayashi1959}, this geometric explanation applies to the case of complex manifolds.


\subsection{The Bergman kernel form on a complex manifold}

Let $M$ be an $n$-dimensional complex manifold and $A^2(M)$ the space of holomorphic $n$-forms $f$ on $M$ satisfying 
$$
\Big|\int_Mf\wedge\overline{f}\Big|<\infty.
$$ 
Let $\{\phi_0,\phi_1,\phi_2,\ldots\}$ be a complete orthonormal system for the Hilbert space $A^2(M)$ and $\overline{M}$ the complex manifold conjugate to $M$. Define the holomorphic $2n$-form on $M\times\overline{M}$ by
$$
K(z,\overline{w})=\sum_{j\geq0}\phi_j(z)\wedge\overline{\phi_j(w)}.
$$
This construction is independent of the choice of orthonormal basis. Using the diagonal embedding $\iota:M\hookrightarrow M\times\overline{M}$, defined by $\iota(z)=(z,\overline{z})$, and the natural identification of $M$ with $\iota(M)$, $K(z,\overline{z})$ can be considered as a $2n$-form on $M$. This is called the {\it Bergman kernel form} of $M$.

Consider the case that the Bergman kernel form is non-zero at any point of $M$. In a local coordinate system $(U,(z_1,\ldots,z_n))$, the Bergman kernel form can be written as
$$
K(z,\overline{z})=K^{\ast}_U(z,\overline{z})dz_1\wedge\cdots\wedge dz_n\wedge d\overline{z_1}\wedge\cdots\wedge d\overline{z_n},
$$ 
where $K^{\ast}_U(z,\overline{z})$ is a well-defined function on $U$. Set
$$
ds^2_M:=\sum\limits_{j,k=1}^n g_{j\overline{k}}(z)dz_j\otimes d\overline{z_k}=\sum\limits_{j,k=1}^n \frac{\partial^2\log K^{\ast}_U(z,\overline{z})}{\partial z_j\partial\overline{z_k}}dz_j\otimes d\overline{z_k}.
$$
This is independent of the choice of local coordinate system. When the matrix $G(z):=(g_{j\overline{k}}(z))$ is positive-definite for each $z\in M$, $ds^2_M$ is called the {\it Bergman metric} of $M$.


\subsection{Bergman representative coordinates}

From now on, suppose that $M$ is a complex manifold which possesses the Bergman metric. (In fact, many non-compact complete K\"{a}hler manifold with negative curvature admit the Bergman metric \cite{GW1979}, Theorem H). In a local coordinate system $(U\times \overline{V},(z_1,\ldots,z_n,\overline{w_1},\ldots,\overline{w_n}))$ for $M\times\overline{M}$, 
$$
K(z,\overline{w})=K^{\ast}_{U\times\overline{V}}(z,\overline{w})dz_1\wedge\cdots\wedge dz_n\wedge d\overline{w_1}\wedge\cdots\wedge d\overline{w_n},
$$
where $K^{\ast}_{U\times\overline{V}}(z,\overline{w})$ is a well-defined function on $U\times \overline{V}$. Given a point $\overline{p}\in\overline{V}$, define the following holomorphic coordinate system centered at $p$ (cf. \cite{davidov1977},\cite{GKK2011}, and \cite{dinew2011}).

\begin{defin}
The {\it Bergman representative coordinate system} at $p$ is defined by 
$$
\hbox{rep}_p(z)=(\zeta_1(z),\ldots,\zeta_n(z)),
$$
where: 
$$
\zeta_j(z):={g}^{\overline{k}j}(p)\Big\{\frac{\partial}{\partial \overline{w_k}}\Big|_{w=p}\log K^{\ast}_{U\times\overline{V}}(z,\overline{w})-\frac{\partial}{\partial \overline{w_k}}\Big|_{w=p}\log K^{\ast}_{V\times\overline{V}}(w,\overline{w})\Big\}.
$$
\end{defin}

\begin{rem}
The above construction is independent of the choice of a coordinate system $(U,(z_1,\ldots,z_n))$ for $M$. But it depends on the choice of local coordinate system $(V,(w_1,\ldots,w_n))$. Note that $\hbox{rep}_p$ extends to a almost global function, well-defined on the whole of $M$ except for the analytic variety $Z^p_0:=\{z\in M:K(z,\overline{p})=0\}$ (this set is well-defined because of the transformation formula \ref{trk} in section 4.1).
\end{rem}


\section{Bochner's normal coordinates and connection}


For real analytic K\"{a}hler manifolds, Bochner constructed the K\"{a}hler normal coordinate system, a version of the representative coordinate system from the K\"{a}hler potential \cite{bochner1947}. This normal coordinate system is strongly related to the exponential map of the K\"{a}hler metric \cite{Demailly1982}. We feel that this relation can be better explained via the language of vector bundles and connections \cite{kapranov1999}. Therefore, we reorganize this information, scattered in the literature.


\subsection{Bochner's normal coordinates}

Suppose that $M$ is a K\"{a}hler manifold with the real analytic K\"{a}hler metric $g$. In \cite{bochner1947}, a K\"{a}hler normal coordinate system is defined as follows: 

\begin{prop}[Bochner's normal coordinates]
Given $p\in M$, there exist holomorphic coordinates $(\zeta_1,\ldots,\zeta_n)$, unique up to unitary linear transformations satisfying
\begin{enumerate}
\item [(i)] $\zeta(p)=0$,\smallskip
\item [(ii)] $g_{j\overline{k}}(p)=\delta_{jk}$,\smallskip
\item [(iii)] $dg_{j\overline{k}}(p)=0$,\smallskip
\item [(iv)] $\frac{\partial^Ig_{j\overline{k}}}{\partial \zeta_1^{i_1} \cdots\partial \zeta_n^{i_n}}(p)=0$, for all $I\geq 1$ and $i_1+\cdots+i_n=I$.
\end{enumerate}
\end{prop}

In \cite{bcov1994}, Bochner's coordinate system was rediscovered in the context of mathematical physics. There, the Bochner coordinates were called the {\it canonical coordinates}. Their result is 

\begin{prop}[Bershadsky, Cecotti, Ooguri and Vafa \cite{bcov1994}]
Bochner's normal coordinates $(\zeta_1,\ldots,\zeta_n)$ can be expressed in terms of the K\"{a}hler potential $\psi(z,\overline{z})$:
$$
\zeta_j(z)=\sqrt{g}^{\overline{k}j}(p)\Big\{\frac{\partial}{\partial \overline{w_k}}\Big|_{w=p}\psi(z,\overline{w})-\frac{\partial}{\partial \overline{w_k}}\Big|_{w=p}\psi(w,\overline{w})\Big\},
$$
where ${\sqrt{g}}^{\overline{k}j} (p)$ is as follows: Since $G(p):=(g_{j\overline{k}}(p))$ is a positive-definite Hermitian matrix, there exists a matrix $A$ such that $G(p)=A \overline{A^t}$. Denote by ${\sqrt{g}}^{\overline{k}j}(p)$ the $(k,j)$-th entry of the inverse matrix of $A$.
\end{prop}

\begin{cor}
Bochner's normal coordinate system for a manifold with the Bergman metric is the same as the Bergman representative coordinate system of the K\"{a}hler potential $\log K(z,\overline{z})$ up to the normalization factor $\sqrt{g}^{\overline{k}j}(p)$.
\end{cor}


We explain how to separate the holomorphic part from the Riemannian exponential map and show that it coincides with the inverse map of the Bochner normal coordinate system. Our exposition follows those of \cite{Demailly1982}, \cite{Demailly1994}, and \cite{kapranov1999}.


\subsection{The holomorphic exponential map}

Let $M$ be a real-analytic K\"{a}hler manifold. The construction of the holomorphic exponential map from the Riemannian exponential map $\text{exp}_p:T_pM\rightarrow M$ consists of two steps: {\small\bf (1) complexification}, {\small\bf (2) restriction.}\medskip\\
{\small\bf Step 1. Complexification (Polarization).} Note that the real-analytic manifold $M$ of real dimension $n$ can be embedded to become a totally real submanifold of the complex manifold $\mathbb{C}M$ of complex dimension $n$.

\begin{thm}[Whitney-Bruhat \cite{Whitney1959}]
Every real analytic manifold $M$ can be embedded to become a totally real submanifold of a complex manifold. This embedding is unique in the sense that, if $\iota_1:M\hookrightarrow\mathbb{C}M_1$ and $\iota_2:M\hookrightarrow\mathbb{C}M_2$ are such embeddings (to the spaces of same dimension), then there exist neighborhoods $U_1$ and $U_2$ of $M$ in $\mathbb{C}M_1$ and $\mathbb{C}M_2$ respectively, and a biholomorphism $f:U_1\rightarrow U_2$ such that $\iota_2=f\circ\iota_1$.
\end{thm}

Since $M$ is a K\"{a}hler manifold (of course complex manifold) in our setting, we have two embeddings, the complexification $T_pM\hookrightarrow T_p^{\mathbb{C}}M$ and the diagonal embedding $\iota:M \hookrightarrow M\times\overline{M}$ as the embedding in Theorem 3.4. Then, apply the following lemma to the exponential map $\text{exp}_p:T_pM\rightarrow M$.

\begin{lem}
Let $M$ and $N$ be totally real submanifolds of complex manifolds $\mathbb{C}M$ and $\mathbb{C}N$, and $f:M\rightarrow N$ a real-analytic diffeomorphism. Then there are neighborhoods $U$ and $V$ of $M$ and $N$, and a unique holomorphic map $f^{\mathbb{C}}:U\rightarrow V$ extending $f$.
\end{lem}

\begin{proof}
See the proof of Lemma 1.2.1 in \cite{Stenzel1990}.
\end{proof}

Denote by $\text{exp}_p^{\mathbb{C}}$ the unique holomorphic extension of $\text{exp}_p$.\medskip\\
{\small\bf Step 2. Restriction.} Use the decomposition $T_p^{\mathbb{C}}M\cong T'_pM\oplus T_p''M$ where $T'_pM$: the holomorphic tangent space and $T''_pM$: the anti-holomorphic tangent space of the given real-analytic K\"{a}hler manifold $M$. Then restrict the complexified map $\text{exp}_p^{\mathbb{C}}$ to $T'_pM$.

\begin{defin}
The restriction map $\text{exp}_p^{\mathbb{C}}\big|_{T'_pM}(\zeta):=\text{exp}_p^{\mathbb{C}}(\zeta,0)$ is called the {\it holomorphic exponential map}.
\end{defin}

We remark that the above definition is the same as the following definition which first appeared in \cite{Demailly1982}.

\begin{defin}
Take the power series expansion of the exponential map of the K\"{a}hler metric $\text{exp}_p:T'_pM\oplus T_p''M\rightarrow M$ and the decomposition
$$
\text{exp}_p(\zeta,\overline{\zeta})=f(\zeta) + g(\zeta,\overline{\zeta}),
$$
on some neighborhood of $0$, where $f$ is holomorphic in $\zeta$ and $g$ is the sum of all monomials in the power series expansion of $\text{exp}_p(\zeta,\overline{\zeta})$ which are not holomorphic in $\zeta$. Then {\it the holomorphic part} of the exponential map at $p$ is defined to be 
$$
\text{exph}_p(\zeta):=f(\zeta).
$$
\end{defin}


\subsection{The Bochner connection}\label{bc}

We present the construction of the holomorphic affine connection, whose affine exponential map is the {\it holomorphic exponential map} $\text{exph}_p$. We also show that $\text{exph}_p$ is the same as the inverse to the Bochner normal coordinate system, using the affine geodesic equations of the connection.

\begin{thm} [Kapranov \cite{kapranov1999}]
There exists a holomorphic affine connection $\nabla^{\mathbb{C}}$ on $T'(M\times\overline{M})$, defined over a neighborhood of $\iota(M)$, whose affine exponential map is $\text{\rm exp}_p^{\mathbb{C}}$. The restriction of $\nabla^{\mathbb{C}}$ to $T'_pM$ is also a holomorphic affine connection, defined over a neighborhood of $p$ in $M$. The affine exponential map of $\nabla^\mathbb{C}|_{T'_pM}$ is $\text{\rm exph}_p$.
\end{thm}

\begin{proof}
Let $\nabla$ be the K\"{a}hler connection, defined by the Christoffel symbols $\Gamma^j_{kl}(z,\overline{z})=\frac{\partial g_{k\overline{m}}(z,\overline{z})}{\partial z_l}g^{\overline{m}j}(z,\overline{z})$. Denote by $\nabla^{\mathbb{C}}$ the analytic continuation (complexification) of $\nabla$. Then $\nabla^{\mathbb{C}}$ is an affine connection, defined by the coefficients of the connection 1-form:
$$
\Gamma^j_{kl}(z,\overline{w})=\frac{\partial g_{k\overline{m}}(z,\overline{w})}{\partial z_l}g^{\overline{m}j}(z,\overline{w}),
$$
where $(z,\overline{w})$ are holomorphic coordinates for $M\times\overline{M}$. Moreover, its affine exponential map is the same as $\text{exp}_p^{\mathbb{C}}$, since the complexification of $\text{exp}_p$ is unique. 

To prove the second statement, take the decomposition
$$
T'_{(p,\overline{p})}(M\times\overline{M})=T'_pM\oplus T'_{\overline{p}}\overline{M},
$$
and restrict $\nabla^\mathbb{C}$ to $T'_pM$. Then this is a holomorphic affine connection on $T'_pM$, defined only in some neighborhood of $p$ in $M$. The affine exponential map of this connection is the holomorphic exponential map $\text{exph}_p$.
\end{proof}

From now on, we denote $\nabla^\mathbb{C}|_{T'_pM}$ by $\nabla^p$, and call it the {\it Bochner connection} at $p$. The following lemma shows the affine geodesic equations for the holomorphic exponential map $\text{exph}_p$.

\begin{lem}
The geodesic of the Bochner connection $\nabla^p$ emanating from $p$ in the initial direction $\zeta\in T'_pM$ satisfies the following system of second order ODE:
\begin{equation}\label{eqn}
\left\{ \begin{array}{l}
\frac{d^2z_{j}(t)}{dt^2}+\Gamma^j_{kl}(z(t),\overline{p})\frac{dz_k(t)}{dt}\frac{dz_l(t)}{dt}=0,\medskip\medskip\\ 
\ \ \ \ \ z(0)=p,\ \ \ \frac{dz_{j}}{dt}(0)=\zeta_j.
\end{array}\right.
\end{equation}
\end{lem}

\begin{proof}
The curve $\text{exp}_p^{\mathbb{C}}(\zeta t,\overline{\xi} t)$, constructed by the affine exponential map of $\nabla^\mathbb{C}$ satisfies
\begin{equation}\label{eqn1}
\left\{ \begin{array}{l}
\frac{d^2z_j(t)}{dt^2}+\Gamma^j_{kl}(z(t),\overline{w}(t))\frac{dz_k(t)}{dt}\frac{dz_l(t)}{dt}=0,\medskip\medskip\\ 
\ \ \ \ \ z(0)=p,\ \ \ \frac{dz_{j}}{dt}(0)=\zeta_j,
\end{array}\right.
\end{equation}
and
\begin{equation}\label{eqn2}
\left\{ \begin{array}{l}
\frac{d^2\overline{w_j}(t)}{dt^2}+\Gamma^{\overline{j}}_{\overline{k}\overline{l}}(z(t),\overline{w}(t))\frac{d\overline{w_k}(t)}{dt}\frac{d\overline{w_l}(t)}{dt}=0,\medskip\medskip\\ 
\ \ \ \ \ \overline{w}(0)=\overline{p},\ \ \ \frac{d\overline{w_j}}{dt}(0)=\overline{\xi_j},
\end{array}\right.
\end{equation}
where $(\zeta,\overline{\xi})\in T'_pM\oplus T''_pM=\mathbb{C}T_pM$. It suffices to let $\xi\equiv0$, since the solution of (\ref{eqn2}) becomes the constant map $(w_1,\ldots,w_n)\equiv(p_1,\ldots,p_n)$.
\end{proof}

Using the above lemma, we prove the follwing proposition, which appeared first in \cite{kapranov1999}.

\begin{prop}\label{geo}
The inverse to the holomorphic exponential map at $p$ of the real analytic K\"{a}hler metric is the Bochner normal coordinate system at $p$, up to unitary linear transformations.
\end{prop}

\begin{proof}
Let $\varphi$ be the inverse to the Bochner normal coordinate system and $\tilde{\gamma}(t)$ the curve in $M$ given by $\tilde{\gamma}(t)=\varphi(v t)$ where $v\in\mathbb{C}^n\cong T'_pM$. It is enough to show that $\tilde{\gamma}(t)=(z_1(t),\ldots,z_n(t))$ satisfies (\ref{eqn}). By the definition of the normal coordinates, we obtain $\tilde{\gamma}(0)=\varphi(0)=p$ and 
\begin{equation}\label{atz}
\frac{\partial \zeta_k}{\partial z_l}=g^{\overline{j}k}(p)g_{l\overline{j}}(z,\overline{p}),\ \ \  \frac{\partial z_k}{\partial \zeta_r}=g_{r\overline{\lambda}}(p)g^{\overline{\lambda}k}(z,\overline{p}).
\end{equation}
Since $\frac{\partial \zeta_k}{\partial z_l}\big|_{z=p}=\delta_{lk}$, $\tilde{\gamma}'(0)=(\frac{dz_1}{dt}(0),\ldots,\frac{dz_n}{dt}(0))=(\frac{d\zeta_1}{dt}(0),\ldots,\frac{d\zeta_n}{dt}(0))=v$.
Then the holomorphicity of the Bochner normal coordinates implies that
\begin{eqnarray*}
\frac{d^2z_{j}(t)}{dt^2} & + & \Gamma^j_{kl}(\tilde{\gamma}(t),\overline{p})\frac{dz_k(t)}{dt}\frac{dz_l(t)}{dt}\\
& = & \frac{\partial^2z_j}{\partial \zeta_r \partial \zeta_s}\frac{d\zeta_r}{dt}\frac{d\zeta_s}{dt}+\Gamma^j_{kl}(\tilde{\gamma}(t),\overline{p})\frac{\partial z_k}{\partial \zeta_r}\frac{d\zeta_r}{dt}\frac{\partial z_l}{\partial \zeta_s}\frac{d\zeta_s}{dt} \\
& = & \Big\{\frac{\partial^2z_j}{\partial \zeta_r \partial \zeta_s}+\Gamma^j_{kl}(\tilde{\gamma}(t),\overline{p})\frac{\partial z_k}{\partial \zeta_r}\frac{\partial z_l}{\partial \zeta_s}\Big\}v_rv_s .
\end{eqnarray*}
Thus it suffices to show that the following analytic differential equations hold: 
\begin{equation}
\frac{\partial^2z_j}{\partial \zeta_r \partial \zeta_s}+\Gamma^j_{kl}(\tilde{\gamma}(t),\overline{p})\frac{\partial z_k}{\partial \zeta_r}\frac{\partial z_l}{\partial \zeta_s}=0.
\end{equation}
The matrix equation $d(A\cdot A^{-1})=0$ and the equation (\ref{atz}) yield
\begin{align*}
\frac{\partial^2z_j}{\partial \zeta_r\partial \zeta_s}&=g_{r\overline{\lambda}}(p)\frac{\partial g^{\overline{\lambda}j}(z,\overline{p})}{\partial z_l}\frac{\partial z_l}{\partial \zeta_s}\\
&=-g_{r\overline{\lambda}}(p)g^{\overline{\lambda}k}(z,\overline{p})\frac{\partial g_{k\overline{m}}(z,\overline{p})}{\partial z_l}g^{\overline{m}j}(z,\overline{p})\frac{\partial z_l}{\partial \zeta_s}\\
&=-\frac{\partial z_k}{\partial \zeta_r}\frac{\partial g_{k\overline{m}}(z,\overline{p})}{\partial z_l}g^{\overline{m}j}(z,\overline{p})\frac{\partial z_l}{\partial \zeta_s}.
\end{align*}
Therefore,  we arrive at 
$$
\frac{\partial^2z_j}{\partial \zeta_r \partial \zeta_s}+\Gamma^j_{kl}(\tilde{\gamma}(t),\overline{p})\frac{\partial z_k}{\partial \zeta_r}\frac{\partial z_l}{\partial \zeta_s}=0.
$$
\end{proof}

\begin{rem}
For a bounded domain with the Bergman metric, it is known that the above analytic equations hold for the Bergman representative map \cite{Lu1984}. This is, of course, strongly analogous to the analysis of flows of vector fields in the context of Riemannian geometry.
\end{rem}


\section{The Bochner connection on a manifold with the Bergman metric} \label{bb}


Let $M$ be a complex manifold with the Bergman metric and $p$ a point in $M$. Since the Bergman metric is a real-analytic K\"{a}hler metric, the Bochner connection can be constructed in an open neighborhood of $p$ as in Section \ref{bc}. On the other hand, we show that the Bochner connection actually extends to the whole manifold except possibly for an analytic variety.

\subsection{The extended Bochner connection}

Suppose that $M$ is a complex manifold which possesses the Bergman metric. Recall that in a local coordinate system $(U\times \overline{V},(z_1,\ldots,z_n,\overline{w_1},\ldots,\overline{w_n}))$ for $M\times\overline{M}$, the Bergman kernel form is
$$
K(z,\overline{w})=K^{\ast}_{U\times\overline{V}}(z,\overline{w})dz_1\wedge\cdots\wedge dz_n\wedge d\overline{w_1}\wedge\cdots\wedge d\overline{w_n},
$$
where $K^{\ast}_{U\times\overline{V}}(z,\overline{w})$ is a well-defined function on $U\times\overline{V}$. Define the tensor on $M\times\overline{M}$ by
$$
G(z,\overline{w}):=\sum\limits_{j,k=1}^n g_{j\overline{k}}(z,\overline{w})dz_j\otimes d\overline{w_k}=\sum\limits_{j,k=1}^n \frac{\partial^2\log K^{\ast}_{U\times V}(z,\overline{w})}{\partial z_j\partial\overline{w_k}}dz_j\otimes d\overline{w_k}.
$$
Let $(\widetilde{U}\times \overline{\widetilde{V}},(\widetilde{z}_1,\ldots,\widetilde{z}_n,\overline{\widetilde{w}_1},\ldots,\overline{\widetilde{w}_n}))$ be another coordinate system. Then, in $(U\times\overline{V})\cap(\widetilde{U}\times\overline{\widetilde{V}})$, the following transformation formula
\begin{equation}\label{trk}
K^{\ast}_{U\times\overline{V}}(z,\overline{w})=K^{\ast}_{\widetilde{U}\times\overline{\widetilde{V}}}(\widetilde{z},\overline{\widetilde{w}}) \det{J_U^{\widetilde{U}}(z)}\overline{\det{J_V^{\widetilde{V}}(w)}}
\end{equation}
holds, where $J_U^{\widetilde{U}}(z)=\big(\frac{\partial \widetilde{z}_k}{\partial z_j}\big)_{n\times n}$ and $J_V^{\widetilde{V}}(w)=\big(\frac{\partial \widetilde{w}_k}{\partial w_j}\big)_{n\times n}$. In terms of matrices,
\begin{equation}\label{trg}
G_{U\times\overline{V}}(z,\overline{w})=J_U^{\widetilde{U}}(z)\cdot\widetilde{G}_{\widetilde{U}\times \overline{\widetilde{V}}}(\widetilde{z},\overline{\widetilde{w}})\cdot\overline{J_V^{\widetilde{V}}(w)}^t
\end{equation}
where $G_{U\times \overline{V}}(z,\overline{w})=\Big[\frac{\partial^2\log K^{\ast}_{U\times \overline{V}}(z,\overline{w})}{\partial z_j\partial\overline{w_k}}\Big]_{n\times n}$, $\widetilde{G}_{\widetilde{U}\times \overline{\widetilde{V}}}(\widetilde{z},\overline{\widetilde{w}})=\Big[\frac{\partial^2\log K^{\ast}_{\widetilde{U}\times\overline{\widetilde{V}}}(\widetilde{z},\overline{\widetilde{w}})}{\partial \widetilde{z}_j\partial\overline{\widetilde{w}_k}}\Big]_{n\times n}$.\\
Given a point $\overline{p}\in\overline{M}$, define the analytic varieties
$$
Z^p_0:=\{z\in M: K(z,\overline{p})=0\} \text{\ \ and\ \ } Z^p_1:=\{z\in M-Z^p_0: \det(G(z,\overline{p}))=0\}.
$$

\begin{lem}
If $f:M\rightarrow \widetilde{M}$ is a biholomorphism with $q=f(p)$, then it satisfies\smallskip
\begin{enumerate}
\item [(1)] $f(Z^p_0)=\widetilde{Z}^q_0$,\smallskip
\item [(2)] $f(Z^p_1)=\widetilde{Z}^q_1$,\smallskip
\item [(3)] $f(M^p)=\widetilde{M}^q$,
\end{enumerate}
where $M^p:=M-(Z^p_0\cup Z^p_1)$ and $\widetilde{M}^q:=\widetilde{M}-(\widetilde{Z}^q_0\cup\widetilde{Z}^q_1)$.
\end{lem}

\begin{proof}
The transformation formulae (\ref{trk}) and (\ref{trg}) prove that these sets are well-defined and invariant under biholomorphisms.
\end{proof}

Let $T'M^p$ be the holomorphic tangent bundle over $M^p$. Then,

\begin{thm}\label{connection}
There exists a holomorphic affine connection $\nabla^p$ on $T'M^p$ satisfying:
\begin{enumerate}
\item [(1)] $\nabla^p$ is locally flat, i.e. the torsion and curvature of $\nabla^p$ are zero.\medskip
\item [(2)] {\rm(\ref{eqn})} are the affine geodesic equations for $\nabla^p$.\medskip
\item [(3)] $f_{\ast}(\nabla^p_XY)=\nabla^q_{\widetilde{X}}{\widetilde{Y}}$  for all $X,Y\in T'M^p$ where $\widetilde{X}=f_{\ast}(X), \widetilde{Y}=f_{\ast}(Y)$ and $f$ is the same as in the preceding lemma.
\end{enumerate}
\end{thm}

\begin{proof}
Define the connection 1-forms as follows: Note that $G:=G_{U\times\overline{V}}(z,\overline{p})$ is an invertible holomorphic $(n\times n)$-matrix on $U\cap M^p$ so that $G^{-1}$ is well-defined on $U\cap M^p$. Define the $(n\times n)$-matrix $\omega$ of holomorphic 1-forms by $\omega:=\partial G\cdot G^{-1}$, locally defined on $U\cap M^p$. In other words, 
$$
\omega_i^j(z)=\Gamma_{ik}^j(z,\overline{p})dz_k=\frac{\partial g_{i\overline{m}}(z,\overline{p})}{\partial z_k}g^{\overline{m}j}(z,\overline{p})dz_k.
$$
Since $\partial G=\omega\cdot G$, the transformation formula (\ref{trg}) yields the \textit{transformation rule for connection 1-forms}: 
\begin{equation}\label{trc}
\omega\cdot J=\partial J+J\cdot\widetilde{\omega}.
\end{equation}

To show (1), observe that $\nabla^p$ is torsion-free, because $\frac{\partial}{\partial z_k}g_{i\overline{j}}=\frac{\partial}{\partial z_i}g_{k\overline{j}}$. Moreover, its curvature form $\Omega:=d\omega-\omega\wedge\omega$ is also zero, because $G:=(g_{j\overline{k}}(z,\overline{p}))$ is holomorphic. More precisely,
\begin{align*}
&d(\partial G\cdot G^{-1})-(\partial G\cdot G^{-1})\wedge(\partial G\cdot G^{-1})\\
=&\partial(\partial G\cdot G^{-1})+\partial G\wedge\partial G^{-1}\cdot G\cdot G^{-1}\\
=&-\partial G\wedge\partial G^{-1}+\partial G\wedge\partial G^{-1}=0.
\end{align*}

Now, (2) follows from the construction, and (3) follows by (\ref{trc}).
\end{proof}

\begin{rem}\label{remc}
The last statement in Theorem \ref{connection} implies the $\mathbb{C}$-linearity of the Bergman representative coordinates as follows: Since geodesics are straight lines in this coordinate system, $\hbox{\rm rep}_{f(p)}\circ f\circ \hbox{\rm rep}_p^{-1}$ maps straight lines to straight lines. Thus it is $\mathbb{R}$-linear. Since the representative map is holomorphic, this is $\mathbb{C}$-linear. This is the geometric proof of Theorem \ref{clin}, promised in Section \ref{br}.
\end{rem}

It is possible to find the formula of the inverse to the affine exponential map of $\nabla^p$ not only at $p$ but also at an arbitrary point $q\in M^p$. The proof is the same as that of Theorem \ref{geo}.

\begin{prop}
Denote by $\text{\rm exph}_q$ the affine exponential map of $\nabla^p$ at $q$. Let $(\zeta_1,\ldots,\zeta_n)$ be the coordinate system for the holomorphic tangent space at $q$, $T'_qM^p$. Then, in the local coordinate neighborhood $(U,(z_1,\ldots,z_n))$ containing $q$, 
$$
\text{\rm exph}_q^{-1}(z)=(\zeta_1(z),\ldots,\zeta_n(z)),
$$
where:
$$
\zeta_j(z)=\sqrt{g}^{\overline{k}j}(q,\overline{p})\Big\{\frac{\partial}{\partial \overline{w_k}}\Big|_{w=p}\log K^{\ast}_{U\times\overline{V}}(z,\overline{w})-\frac{\partial}{\partial \overline{w_k}}\Big|_{w=p}\log K^{\ast}_{U\times\overline{V}}(q,\overline{w})\Big\}.
$$
\end{prop}

\begin{rem}
The above proposition implies that ${\hbox{\rm exph}_q}^{-1}$ is a linear transform of ${\hbox{\rm exph}_p}^{-1}$. Therefore $M^p$ can be covered by copies (by linear maps) of the representative coordinates. This can be explained by the concept of affine structures in \cite{matsushima1968}. This concept will be introduced in the following subsection.
\end{rem}
\begin{rem}
For a real-analytic K\"{a}hler manifold, it is known that such an affine structure exists on a local neighborhood of a given point [24].
\end{rem}


\subsection{Affine structure of $M^p$}

The proof of Theorem \ref{geo} also implies the following

\begin{prop}
Let $U$ be a local neighborhood of $p$ and $V$ a local neighborhood of $0$ such that ${\hbox{\rm exph}_p}^{-1}:U\rightarrow V$ is biholomorphic. Take any straight line $l$ in $V$ (not necessarily passing through $p$). Then ${\hbox{\rm exph}_p(l)}$ is a geodesic of $\nabla^p$.
\end{prop}

This proposition follows immediately from the affine structures in the below:

\begin{defin}
Let $X$ be a complex manifold of dimension $n$ and $\mathcal{M}=\{U_i,\phi_i\}_{i\in I}$ the maximal atlas. A subset $\mathcal{A}=\{U_j,\phi_j\}_{j\in J},J\subset I,$ of $\mathcal{M}$ is called an {\it affine atlas} of $X$ if all transition maps are complex affine transformation of $\mathbb{C}^n$. We say that each maximal affine atlas defines a complex {\it affine structure} of $X$.
\end{defin}

\begin{thm} [Gunning \cite{Gunning1967}, Matsushima \cite{matsushima1968}]
There is a one-to-one correspondence between the set of all complex affine structures on a complex manifold $X$ and the set of all locally flat holomorphic affine connections on $X$.
\end{thm}

\begin{rem}
For any $x,y\in M^p$, $\text{exph}_y\circ\text{exph}_x^{-1}$ is an affine transformation of $\mathbb{C}^n$. Thus $M^p$ has a complex affine structure and the Bochner connection $\nabla^p$ is the corresponding locally flat holomorphic affine connection.
\end{rem}


\section{Geodesics of the Bochner connection $\nabla^p$}

\subsection{Incompleteness of $\nabla^p$}

The behavior of geodesics of $\nabla^p$ played an important role in the proof of the following theorem, which generalizes Fefferman's extension theorem.

\begin{thm}[Webster \cite{Webster1979}]
Let $f:\Omega\rightarrow\widetilde{\Omega}$ be a biholomorphism between bounded domains with smooth boundaries. Suppose that their Bergman kernels are smooth up to the boundaries. Then $f$ extends smoothly to a dense open subset of $\partial\Omega$.
\end{thm}

\begin{rem}
In the proof, he used incompleteness of $\nabla^p$, that is, geodesics of $\nabla^p$ extend through the boundary points.
\end{rem}


On the other hand, one might expect to find a suitable K\"{a}hler metric, compatible with $\nabla^p$. But this is impossible because the exponential map of a K\"{a}hler metric is holomorphic if and only if the metric is flat (the Euclidean metric). However, using the image of geodesics under $\text{rep}_p=\text{exph}_p^{-1}$, it is possible to define a distance between two points in a connected manifold $M^p$.

\begin{defin} [Intrinsic distance]
Let $M$ be a connected manifold with the Bergman metric and $\mathcal{A}=\{U_i,\phi_i\}_{i\in I}$ the affine structure of $M^p$, given by the Bochner normal coordinate system. If $x,y\in U_i$ for some $U_i$, then define $\delta^p(x,y)$ to be the euclidean norm of the vector 
$$
\Big(\ldots,\frac{\partial}{\partial \overline{w_k}}\Big|_{w=p}\log K(x,\overline{w})-\frac{\partial}{\partial \overline{w_k}}\Big|_{w=p}\log K(y,\overline{w}),\ldots\Big).
$$
In general, if $x,y$ are arbitrary points in $M^p$, then define the {\it intrinsic distance} by
$$
\text{d}^p(x,y):=\inf\sum\limits_{j=1}^N\delta^p(p_{j-1},p_j),
$$
where the infimum is taken over all possible partitions with $x=p_0,\ldots,p_N=y$.

This is well-defined since there always exists a broken geodesic between two points in the connected affine manifold $(M^p,\nabla^p)$.
\end{defin}

\begin{rem}
For the symmetry $\text{d}^p(x,y)=\text{d}^p(y,x)$, we do not use the normalization factor of the Bochner normal coordinate system. Although $\text{d}^p$ is not a biholomorphic invariant, its finiteness between two points is a biholomorphic invariant. 
\end{rem}

The following theorem shows the relation between the intrinsic distance $\text{d}^p$ and the analytic variety $Z_0^p$.

\begin{thm}
If $q\in Z_0^p$, then there are no geodesics toward $q$ such that the intrinsic distance is finite.
\end{thm}

\begin{proof}
Suppose that there exists a geodesic toward $q$ such that the intrinsic distance is finite. Then the Bochner normal coordinate system is well-defined at $q$. Note that $\frac{\partial}{\partial\overline{w_k}}K(q,\overline{w})/K(q,\overline{w})$ is an anti-holomorphic function in $w$ for each $k$. Fix an index $k$ and then in the $w_k$-section, this is a function of one variable. Since $K(q,\overline{p})=0$, it has a simple pole at $p$ so that the value is infinite. This implies that $\delta^p$ is infinite, which contradicts the finiteness of the intrinsic distance.
\end{proof}

\subsection{Examples}
Let $\Omega$ be a bounded domain in $\mathbb{C}^n$. Denote the Levi-Civita connection of the Bergman metric of $\Omega$ by $\nabla$. Given a point $p\in\Omega$, denote the Bochner connection by $\nabla^p$. In this section, we show the difference between two connections $\nabla$ and $\nabla^p$ by comparing geodesics of two connections.

\subsubsection{Unit ball}

Let $\mathbb{B}^n$ be the unit ball in $\mathbb{C}^n$. Then the Riemannian exponential map of $\mathbb{B}^n$ at $0$ is
$$
{\rm exp}_0(\zeta)=\frac{\tanh(\abs{\zeta})}{\abs{\zeta}}\cdot\zeta=\Big(\frac{\tanh(\abs{\zeta})}{\abs{\zeta}}\zeta_1,\ldots,\frac{\tanh(\abs{\zeta})}{\abs{\zeta}}\zeta_n\Big),
$$
where $\zeta=(\zeta_1,\ldots,\zeta_n)$ is the standard complex coordinate for $\mathbb{C}^n\cong T_0\mathbb{B}^n$.\\
On the other hand, the representative map at $0$ is
$$
{\rm rep}_0(z)=\sqrt{n+1}\cdot z=(\sqrt{n+1}z_1,\ldots,\sqrt{n+1}z_n).
$$
Therefore, this shows that in the unit ball, geodesics of $\nabla$ and $\nabla^0$ at $0$ move along the same direction but with different speed.

\subsubsection{Unit polydisk}

Let $\mathbb{D}^n$ be the unit polydisk in $\mathbb{C}^n$. Then the Riemannian exponential map of $\mathbb{D}^n$ at $0$ is
$$
{\rm exp}_0(\zeta)=\Big(\frac{\tanh(\abs{\zeta_1})}{\abs{\zeta_1}}\zeta_1,\ldots,\frac{\tanh(\abs{\zeta_n})}{\abs{\zeta_n}}\zeta_n\Big),
$$
where $\zeta$ is the standard complex coordinate for $\mathbb{C}^n\cong T_0D^n$.\\
On the other hand, the representative map at $0$ is
$$
{\rm rep}_0(z)=\sqrt{2}\cdot z=(\sqrt{2}z_1,\ldots,\sqrt{2}z_n).
$$
Unlike the unit ball, this shows that the geodesics of $\nabla$ and $\nabla^0$ at $0$ move along different directions with different speed.

\subsubsection{Skwarczy\'nski's annulus}

Let $A:=\{z\in\mathbb{C}:0<r<|z|<1\}$. Then the Bergman kernel of $A$ is
$$
K(z,\overline{w})=\frac{\wp(\log z\overline{w})+\frac{\eta_1}{\omega_1}}{\pi z\overline{w}},
$$ 
where $\wp$ is the Weierstrass elliptic function with half periods $\omega_1=\log(1/r)$, $\omega_2=\pi i$ and $\eta_1$ is the increment of the Weierstrass zeta function $\zeta$ with respect to $\omega_1$.

The geodesics of the Levi-Civita connection $\nabla$ were already studied in \cite{Herbort1983}. To study the geodesics of the Bochner connection $\nabla^p$, we need the following:

{\small \bf Zeros of $K(z,\overline{p})$:}
Define $h(\lambda):=\wp(\log\lambda)+\frac{\eta_1}{\omega_1}$ on the set $\widetilde{A}:=\{\lambda\in\mathbb{C}:r^2<|\lambda|<1\}$. Then $K(z,\overline{w})=\frac{h(\lambda)}{\pi\lambda}$, where $\lambda=z\overline{w}$. In \cite{Skwarczynski1969}, Skwarczy\'nski proved that
\begin{itemize}
\item $h(\lambda)$ is real, $\forall\lambda\in\mathbb{R}$.\smallskip
\item For $r<e^{-2}$, there exists a point $\lambda\in\widetilde{A}$ such that $h(\lambda)=0$.
\end{itemize}
Later, B{\l}ocki improved the above result as follows (cf. \cite{blocki2010}, Theorem 3.4):
\begin{itemize}
\item $h(-1)=h(-r^2)<0$ and $h(-r)>0$.\smallskip
\item For $r<1$, there exist only two solutions $\lambda_1, \lambda_2$ of the equation $h(\lambda)=0$ in $\widetilde{A}=\{\lambda\in\mathbb{C}:r^2<|\lambda|<1\}$, where $\lambda_2\in(-1,-r)$ and $\lambda_1\in(-r,-r^2)$.
\end{itemize}
Fix a point $p\in A$. The symmetry of the annulus allows one to assume that $p\in(r,1)$ on the real line. Let $\lambda_1^p$ and $\lambda_2^p$ be the solutions of the equation $h(z\overline{p})=0$ satisfying $\lambda_1=\lambda_1^p\overline{p}$ and $\lambda_2=\lambda_2^p\overline{p}$. Since $\lambda_2^p\in(-1/p,-r/p)$, $\lambda_1^p\in(-r/p,-r^2/p)$, the number of elements of the zero set $\{z\in A:K(z,\overline{p})=0\}$ depends on the location of $p$ (For details, see Corollary 2.5 in \cite{Jacobson2012}). For example, if $p$ is close enough to $1$ (or $r$), then there exists only one solution of the equation $K(z,\overline{p})=0$ in the annulus $A$, located in $(-1,-r)$ (or $(-r,-r^2)$).

{\small \bf Geodesics of the Bochner connection $\nabla^p$:}
Recall that the exponential map of the Bochner connection $\nabla^p$ is the inverse map of $\hbox{\rm rep}_p$. A simple computation shows that
$$
\hbox{\rm rep}_p(z)=C_1\cdot\frac{\wp'(\log z\overline{p})}{\wp(\log z\overline{p})+\frac{\eta_1}{\omega_1}}+C_2,
$$ 
where $C_1$ and $C_2$ are constants, and $\wp'$ is the first derivative of the Weierstrass elliptic function $\wp$. Then $\hbox{\rm rep}_p$ is also an elliptic function that shares the same periods with $\wp$ and possesses three simple poles. Since the image of any elliptic function contains $\mathbb{C}$, the pre-image of each straight line consists of three curves (counting multiplicity) in the annulus $\{z\in\mathbb{C}:r^2<|z\overline{p}|<1\}$.

Note that the geodesics of $\nabla^p$ are the images of straight lines by the holomorphic exponential map $\hbox{\rm exph}_p=\hbox{\rm rep}_p^{-1}$. Therefore, it is enough to study the pre-images of straight lines by the elliptic function $f(\lambda):=\frac{\wp'(\lambda)}{\wp(\lambda)+c}$, where the lattice of periods is $\Lambda=\{\mathbb{Z}(2\omega_1)+\mathbb{Z}(2\omega_2)\}$. Since $\omega_1\in\mathbb{R}$ and $\omega_2\in i\mathbb{R}$, the function $\wp$ and its derivative $\wp'$ have rectangular lattice so that $\wp(z)=\overline{\wp(\overline{z})}, \wp'(z)=\overline{\wp'(\overline{z})}$. This implies that for all $t\in\mathbb{R}$,
\begin{itemize}
\item $\wp(t2\omega_j)=\overline{\wp(t\overline{2\omega_j})}=\overline{\wp(t2\omega_j)}$ and $\wp(\omega_i+t2\omega_j)=\overline{\wp(\omega_i+t2\omega_j)}$,\smallskip
\item $\wp'(t2\omega_j)=\pm\overline{\wp'(t2\omega_j)}$ and $\wp'(\omega_i+t2\omega_j)=\pm\overline{\wp'(\omega_i+t2\omega_j)}$.
\end{itemize}
Since $c=\frac{\eta_1}{\omega_1}\in\mathbb{R}$, we see that
\begin{itemize}
\item $f(\mathbb{R})\subset\mathbb{R}$ and $f(i\mathbb{R})\subset i\mathbb{R}$.\smallskip
\item $f(\mathbb{R}+\omega_2)\subset\mathbb{R}$ and $f(\omega_1+i\mathbb{R})\subset i\mathbb{R}$.
\end{itemize} 
Therefore, $f^{-1}(\mathbb{C}-(\mathbb{R}\cup i\mathbb{R}))$ consists of 4 open sub-rectangles which are divided by $\omega_1$ and $\omega_2$ in the fundamental region $\{z\in\mathbb{C}:0\leq Re(z)\leq 2\omega_1,0\leq Im(z)\leq Im(2\omega_2)\}$. This provides an approximate, but useful, information on the location of geodesics.


\section{On the variety $Z^p_0\cup Z^p_1$}

Let $\Omega$ be a bounded domain in $\mathbb{C}^n$ and $K(z,\overline{w})$ the Bergman kernel of $\Omega$. Fix a point $p\in\Omega$. Then the Bochner connection $\nabla^p$ can be constructed over $\Omega^p=\Omega-(Z^p_0\cup Z^p_1)$. Whether $Z^p_0$ is empty is related to the well-known Lu Qi-Keng conjecture; a bounded domain $\Omega$ is called a {\it Lu Qi-Keng domain} if $Z^p_0$ is empty for each point $p\in\Omega$.

\begin{rem}
It was anticipated in 1960's that every bounded domain should be a Lu Qi-Keng domain, called the Lu Qi-Keng conjecture. However, many counterexamples have been discovered (\cite{Skwarczynski1969}, \cite{boas1986,boas1996}, and others), and in contrast, many domains are Lu Qi-Keng.
\end{rem}




On the other hand, note that
$$
\hbox{\rm det}[G(z,\overline{p})]=\frac{\hbox{\rm det}[K(z,\overline{p})\frac{\partial^2}{\partial z_i \partial \overline{w_j}}\big|_{w=p} K(z,\overline{w})-\frac{\partial}{\partial z_i}K(z,\overline{p})\frac{\partial}{\partial\overline{w_j}}\big|_{w=p}K(z,\overline{w})]}{K(z,\overline{p})^{2n}}.
$$
Set
$$
\widehat{Z}^p_1:=\Big\{z\in\Omega:\hbox{\rm det}\Big[K(z,\overline{p})\frac{\partial^2 K(z,\overline{w})}{\partial z_i \partial \overline{w_j}}\Big|_{w=p} -\frac{\partial K(z,\overline{p})}{\partial z_i}\frac{\partial K(z,\overline{w})}{\partial\overline{w_j}}\Big|_{w=p}\Big]=0\Big\}.
$$
It is known that if $n>1$, then $Z^p_0\subset \widehat{Z}^p_1$, and the statement is false if $n=1$ (cf. \cite{dinew2011}, Lemma 2.1). In particular, ${Z^p_0} \nsubseteq \widehat{Z}^p_1$ for the annulus in the complex plane. Therefore, it is natural to ask whether a domain of higher dimension satisfying $Z^p_0\subsetneq \widehat{Z}^p_1$ exists.

\begin{thm}
There exists a smooth bounded strongly pseudoconvex domain $\Omega$ satisfying $Z^p_0\subsetneq \widehat{Z}^p_1$ for some point $p\in\Omega$.
\end{thm}

\begin{proof}
We first prove that when $r$ is small enough, the product domain $A_r\times D$ has a point $p$ such that $Z^p_0\subsetneq \widehat{Z}^p_1$, where $A_r:=\{z\in\mathbb{C}:0<r<|z|<1\}$ is the annulus and $D$ is the unit disk in $\mathbb{C}$.

Let $K((z_1,z_2),(\overline{w_1},\overline{w_2}))=K_{A_r}(z_1,\overline{w_1})K_D(z_2,\overline{w_2})$ be the Bergman kernel of $A_r\times D$. Set 
$$
F_{\Omega}(z,\overline{w}):=\hbox{\rm det}[K_{\Omega}(z,\overline{w})\frac{\partial^2}{\partial z_i \partial \overline{w_j}} K_{\Omega}(z,\overline{w})-\frac{\partial}{\partial z_i}K_{\Omega}(z,\overline{w})\frac{\partial}{\partial\overline{w_j}}K_{\Omega}(z,\overline{w})].
$$ Then
$$
F_{A_r\times D}((z_1,z_2),(\overline{w_1},\overline{w_2}))=K_{A_r}(z_1,\overline{w_1})^2K_D(z_2,\overline{w_2})^2F_{A_r}(z_1,\overline{w_1})F_D(z_2,\overline{w_2}).
$$
Since $K_D(z_2,\overline{w_2})^2F_D(z_2,\overline{w_2})\neq0$ for all $(z_2,\overline{w_2})\in D\times\overline{D}$, it suffices to show that there exists a point $(z,\overline{p})\in A_r\times\overline{A_r}$ satisfying $K_{A_r}(z,\overline{p})\neq 0$ and
$$
K_{A_r}(z,\overline{p})^2F_{A_r}(z,\overline{p})=K_{A_r}(z,\overline{p})^4\partial_z\overline{\partial_w}|_{w=p}\log K_{A_r}(z,\overline{w})=0.
$$
It is known that if $r$ is sufficiently close to $0$ and $p$ is on the real axis, such a point $(z,\overline{p})$ exists, where $z$ is near the imaginary axis (see \cite{dinew2011}, the proof of Theorem 1.5). 

Now we modify this example to the case of irreducible strictly pseudoconvex domains as follows: Consider a strictly pseudoconvex exhaustion $\Omega_j$ for $A_r\times D$. Note that they are irreducible domains (cf. \cite{huckleberry1977}). On the other hand, the Bergman kernel of $\Omega_j$ and its derivatives uniformly converge on compacta to those of $A_r\times D$ (cf. \cite{ramadanov1967}). By Hurwitz's theorem, $\Omega_j$ satisfies ${Z^p_0} \subsetneq \widehat{Z}^p_1$ when $j$ is large enough.
\end{proof}


\section{A generalization of the Lu theorem}

We present an application of the Bochner connection. Let $\Omega$ be a bounded domain in $\mathbb{C}^n$ and $M$ a complex manifold with the positive-definite Bergman metric. Denote their Bergman metric by $\beta_{\Omega}$ and $\beta_M$, respectively. Call the point $p\in\Omega$ a {\it pole of the Bochner connection} $\nabla^p$ whenever $\text{rep}_p:\Omega\rightarrow\mathbb{C}^n$ is one-to-one.

\begin{thm} \label{main}
Suppose that $\Omega$ has a pole $p$ of $\nabla^p$. If there is a surjective holomorphic map $f:\Omega\rightarrow M$ satisfying $f^{\ast}\beta_{M}=\beta_{\Omega}$, then $f$ is a biholomorphism.
\end{thm}

This theorem is a generalization of the following well-known result.

\begin{thm}[Lu Qi-Keng \cite{Lu1966}] \label{lu}
If $\Omega$ is a bounded domain in $\mathbb{C}^n$, whose Bergman metric is complete and has constant holomorphic sectional curvature, then $\Omega$ is biholomorphic to the unit ball.
\end{thm}

\begin{proof} [Proof of $\ulcorner Theorem\ \ref{main} \Rightarrow Theorem\ \ref{lu}\lrcorner$]
Although the below proof is included in the proof of Theorem 4.2.2 in \cite{GKK2011}, we recall that for convinience. Let $c$ be the constant, the holomorphic sectional curvature of $\Omega$. If $c>0$, then $\Omega$ would be a complete Riemannian manifold with all sectional curvatures $\geq c/4>0$. Thus Myers' theorem in Riemannian geometry implies that $\Omega$ is compact, a contradiction. If $c=0$, then the covering space is $\mathbb{C}^n$. Therefore the covering map is constant by Liouville's theorem, which is impossible. Consequently, $c<0$. In that case, it is known that the universal covering space of $\Omega$ is biholomorphic to the unit ball $\mathbb{B}^n$ and the covering map $f:\mathbb{B}^n\rightarrow\Omega$ is a Bergman isometry. Therefore $f$ has to be one-to-one by Theorem \ref{main}, and hence the conclusion of Theorem \ref{lu} follows.
\end{proof}

\begin{rem}
Notice that Theorem \ref{main} does not assume the completeness of the Bergman metric $\beta_{\Omega}$. Moreover, the bounded domain $\Omega$ need not possess the constant holomorphic sectional curvature. Besides the unit ball, the following domains satisfy the hypothesis of Theorem \ref{main}.
\begin{itemize}
\item Every complete circular domain; the center is a pole.
\item Every bounded homogeneous domain; every point is a pole (cf. \cite{Xu1983}).
\end{itemize}
More generally, every bounded domain, which possesses a point $p$ such that the matrix $G(z,\overline{p})$ is independent of $z$, satisfies the hypothesis (the point $p$ is a pole). This is called a {\it representative domain} according to \cite{Lu1984}.
\end{rem}
With slight modification of the proof of Theorem 4.2.2 in \cite{GKK2011}, we present
\begin{proof} [Proof of Theorem \ref{main}]
We are only to show that $f$ is one-to-one. Since $f^{\ast}\beta_{M}=\beta_{\Omega}$ implies that $df$ is non-singular, $f$ is locally invertible. Let $V$ be a neighborhood of $p$ and $U$ a neighborhood of $q:=f(p)$ such that $f|_{V}:V\rightarrow U$ is a biholomorphism. Denote by $g_0$ the inverse of $f|_V$.

On the other hand $\nabla_{\Omega}=f^{\ast}\nabla_{M}$, since $f^{\ast}\beta_{M}=\beta_{\Omega}$, where $\nabla$ denotes the Bergman metric connection. The uniqueness of the polarization and the holomorphicity of $f$ yield that $\nabla_{\Omega}^{p}=f^{\ast}\nabla_{M}^{q}$. This means that $f$ maps geodesics of $\nabla^p$ to geodesics of $\nabla^q$, and one sees that $A:=\text{rep}_q \circ f|_V \circ \text{rep}_p^{-1}$ is $\mathbb{C}$-linear as in Remark \ref{remc}. Thus $g_0:=f|_V^{-1}=\text{rep}_p^{-1}\circ A\circ\text{rep}_{q}$, where $A$ is an invertible $\mathbb{C}$-linear map.

Note that $\text{rep}_q\circ f=A^{-1}\circ\text{rep}_p$ on $\Omega-(Z^p\cup f^{-1}(Z^q))$, where $Z^p:=Z^p_0\cup Z^p_1$ and $Z^q:=Z^q_0\cup Z^q_1$. Then the restriction map $\text{rep}_p^{-1}|_{A\circ\text{rep}_q(M-(f(Z^p)\cup Z^q))}$ is a well-defined holomorphic map. Since the linear map $A$ is everywhere defined and $\text{rep}_q$ extends to a holomorphic mapping of $M-Z^q$, so does $g_0$. Denote by $g$ the extension of $g_0$.

Let $X:=f^{-1}(Z^q)$. Then $g\circ f:\Omega-X\rightarrow\mathbb{C}^n$ is holomorphic and $g\circ f(z)=z$ for every $z\in\Omega-X$. Therefore, for every $\zeta\in M-Z^q$, choose $x\in\Omega$ such that $f(x)=\zeta$. Since $g(\zeta)=g(f(x))=x$, $g(M-Z^q)\subset\Omega$. Note that $g$ is a bounded holomorphic map on the connected manifold $M-Z^q$. By the Riemann extension theorem, $g$ extends to a holomorphic mapping of $M$ into $\mathbb{C}^n$. This shows that $g$ is the left inverse to $f$, and hence $f$ is one-to-one.
\end{proof}


\subsection*{Acknowledgements}
The author would like to express his deep gratitude to Professor Kang-Tae Kim for valuable guidance and encouragements, and to Professor J.P. Demailly for pointing out the relevance of \cite{Demailly1982,Demailly1994}. The author also would like to thank the referees for their valuable comments. This work is part of author's Ph.D. dissertation at Pohang University of Science and Technology.

\end{document}